%% file: main_article.tex
\documentclass[a4paper,12pt,twoside]{article}
\usepackage[inner=3cm,top=3cm,outer=2cm,bottom=3cm]{geometry}
\usepackage[english]{babel}
\usepackage[T1]{fontenc}
\usepackage[latin1]{inputenc}
\usepackage{amssymb}
\usepackage{subfig}
\usepackage{float}
\usepackage{amsmath}
\usepackage{mathtools}
\usepackage{bm}
\usepackage{amsthm}
\usepackage{listings}
\usepackage{graphicx}
\usepackage{color}

\newtheorem{thm}{Theorem}
\newtheorem{lemma}{Lemma}
\newtheorem{cor}{Corollary}

\DeclareMathOperator{\mex}{mex}

\date{\today} 
\title{Analysis of the Odd/odd vertex removal game} 
\author{Oliver Krüger} 

\begin{document} 
\begin{titlepage} \begin{center} 
\Huge \textsc{Analysis of odd/odd vertex removal games on special graphs}\\
$\;$ \\
\large Revised Master Thesis, Royal Institute of Technology - KTH, 2012 \\
\large Oliver Krüger
okruger@math.su.se
\end{center}

\begin{abstract}
We analyze the Odd/odd vertex removal game introduced by P. Ottaway. We
prove that every bipartite graph has Grundy value 0 or 1 only depending
on the parity of the number of edges in the graph, which is a generalization
of a conjecture of K. Shelton. We also answer a question originally posed
by both Shelton and Ottaway about the existance of graphs for every
Grundy value. We prove that this is indeed the case.  
\end{abstract}
\end{titlepage}
\newpage 

\section{Introduction}

In this paper we will analyze some properties of the Odd/odd vertex removal
game introduced by Ottaway\cite{ottaway,art}. In particular we will answer a
question posed by both Ottaway and Shelton\cite{shelton} about the existance
of Odd/odd vertex removal games for every Grundy value.

We will assume that the reader is familiar with the Sprague-Grundy theory
about impartial combinatorial games (see, for example \cite{conwayimpartial} or
\cite{ww}). We will denote the Sprague-Grundy value (also called just Grundy
value) of a game $G$ by
$g(G)$, the disjunctive sum of games $G_1$ and $G_2$ by $G_1 + G_2$. And
the num-sum of two integers $a$ and $b$ by $a \oplus b$. The set of games
where the first player to move will win (when played optimally), i.e. where
$g(G) > 0$, will be denoted $\mathcal{N}$. The complement of that in the
set of all games will be denoted $\mathcal{P}$. We denote by $\mex$ the minimal
excludant function, $\mex(A)$ is the smallest nonnegative integer in
the set $A$.

Remember the main theorem of the Sprague-Grundy theory which states
\begin{thm}
\label{mainthm} \cite{conwayimpartial}
Let $H$ and $G$ be two impartial games. The Grundy value of the disjunctive
sum $G + H$ is
$$g(G + H) = g(G) \oplus g(H)$$
\end{thm}

\subsection{Impartial vertex removal games}

Here follows a brief description of the Odd/odd vertex removal game which
we are about to study including some interesting previous results on
mainly due to Ottaway and Shelton.

In this paper we are going to analyze a vertex removal game that was introduced
by Ottaway. The vertex removal games are played on graphs and each move
consists of removing a vertex and all the edges to and from that vertex. We
decide what vertices we are allowed to remove by some rule concerning the
parity of the degree of the vertices. For these games every position may be
denoted by a graph and we will make no distinction between a graph and a
position in these games.

For example one of the two players may only be allowed to remove edges of even
degree and the other only edges of odd degree. This example is obviously a
partial game since the two players have a different set of options. Games of
the same type but played in digraphs is also possible, but will not be
considered here.

Shelton studies a more general problem than this[6], where
one considers every vertex in a graph to be a coin, i.e. a binary value of
either heads or tails. A player can remove any coin with heads up, and then
flip any adjacent coins (adjacent vertices in the graph). This is more general
than the problem we will study here. This type include the vertex removal game
since if one starts with heads up on every vertex with the parity we want to be
able to take this game is equivalent to the vertex removal game. Furthermore
Shelton proves some results for the Grundy values of these types of games on
particular graphs and that the problem played on directed graphs is
PSPACE-Complete. He also makes two conjectures concerning the odd/odd vertex
removal game which we will study.

We will prove one of these conjectures and
also prove that there are graphs for every Grundy value.

We will be concerned with the impartial type of vertex removal games in
undirected graphs and in particular we will analyze the game where both
players only remove edges of odd degree, the other impartial game has the
following simple charactrization.

\begin{thm}
\cite[Theorem 3.1.3]{ottaway}
When both players can remove only even degree vertices, the game on graph
$G=(V,E)$ is trivial and we have that:
$$g(G) = \left\{
\begin{array}{l l}
0 & \textsl{if} \; |V| \; \textsl{is even} \\
1 & \textsl{if} \; |V| \; \textsl{is odd}
\end{array}
\right.
$$
\end{thm}

This result follows from the fact that any graph with an odd number of
vertices must have at least one even degree vertex. This means that no
position with an odd number of vertices can be terminal. Since $|V|$
alternates between being odd and even the player who always moves to a
position with an even number of vertices will win. Thus we get the theorem
above.

\subsection{Odd/odd vertex removal}
\label{oddoddintro}

The type of game where both players are only allowed to remove vertices
of an odd degree is the main object of study in this paper. We will in this
section present what results are known and in Section \ref{results} we will
show some new results.


\begin{figure}[H]
 \centering
 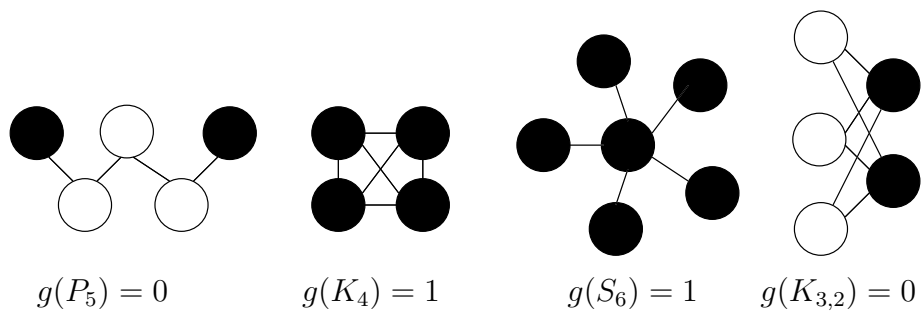
  \caption{Examples of positions with known Grundy values. The vertices
which are removable in the positions are filled in.} 
 \label{figure:examples}
\end{figure}

\begin{thm}
\cite{ottaway} 
\label{hejsan}
When both players can remove only odd degree vertices we have that the
path on $n$ vertices, $P_n$, the complete graph on $n$ vertices, $K_n$
and the star on $n$ vertices, $S_n$ have the same Grundy value
$$g(P_n) = g(K_n) = g(S_n) = \left\{
\begin{array}{l l}
1 & \textsl{if} \; n \; \textsl{is even} \\
0 & \textsl{if} \; n \; \textsl{is odd}
\end{array}
\right.$$ 

The complete bipartite graph $K_{n,m}$ has Grundy value
$$g(K_{n,m}) = \left\{
\begin{array}{l l}
1 & \textsl{if} \; n \; \textsl{and} \; m \; \textsl{are both odd} \\
0 & \textsl{otherwise}
\end{array}
\right.$$
\end{thm}

The classes of path graphs, $P_n$, stars, $S_n$ and complete bipartite
graphs $K_{n,m}$ are all contained in the class of bipartite (not only
complete but arbitrary) graphs. In Section \ref{results} we prove a
result that gives us a simple function for computing the Grundy value
of any bipartite graph. This will also prove a conjecture of
Shelton\cite{shelton} which states that all grid graphs have Grundy value
either 0 or 1, which we will prove is true for any bipartite graph.

For the generalization of the Odd/odd vertex removal games to bipartite graph
where you only consider the parity of the out degree. It is known that unless
$P=PSPACE$ it is computationally difficult to decide the winner given a
a such a game.

\begin{thm}
\cite[Theorem 15]{shelton}
The game of Odd/odd vertex removal on directed graphs is PSPACE-Complete.
\end{thm}

\section{Odd/odd vertex removal analysis}
\label{results}

In this section we will analyze the odd/odd vertex removal game in
some particular graphs. We start by proving a result for bipartite
graphs. Bipartite graphs include such graph classes as paths, grids,
stars, trees and $k$-partite graphs. Recall that one way of defining
a \emph{bipartite} graph is as a graph where every closed trail
has even length\cite[Proposition 1.6.1]{diestel}.

\subsection{Bipartite positions}

We will make use of the following classic theorem by Euler:
\begin{thm}
\cite[Theorem 1.8.1]{diestel}
\label{thm:euler}
A connected graph G has a closed Eulerian trail (sometimes called
Eulerian circuit) if and only if every vertex has even degree.
\end{thm} 

\begin{cor}
\label{cor:terminal}
A position $P$, in an odd/odd vertex deletion game, is terminal if and
only if
$P$ is a graph where every connected component has a closed Eulerian
trail (with the convention that the empty trail, in the single vertex
component, is a closed Eulerian trail).
\end{cor} 

\begin{lemma}
\label{lemma:terminalbipart}
Every terminal position, $P$, in an odd/odd vertex removal game played
on a bipartite graph, $G$, has an even number of edges.
\end{lemma}
\begin{proof}
By Corollary \ref{cor:terminal} each edge in $P$ is part of some closed
Eulerian trail. Since every subraph of a bipartite graph clearly is
bipartite (and by the defintion of a bipartite graph as only having
closed trails of even length)
we know that each such trail must have even
length by noting that a terminal position must be a subgraph of the
initial graph $G$.

Thus each connected component of $P$ has an even number of edges. The
total number of edges is the sum of these.
\end{proof}

\begin{thm}
\label{thm:grundybipart}
Let $G = (V,E)$ be a bipartite graph, then
$$g(G) = \left\{
\begin{array}{l l}
1 & \textsl{if} \; |E| \; \textsl{odd} \\
0 & \textsl{if} \; |E| \; \textsl{even}
\end{array}
\right.$$ 
\end{thm}
\begin{proof}
We need just use Lemma \ref{lemma:terminalbipart}  to see that all
terminal positions have an even number of edges. Thus any position
with an odd number of edges cannot be terminal.

With each move an odd number of edges is removed so the number of edges
remaining in the graph alternate between even and odd with every move.
So one player always moves to a position with an even number of edges
and the other always to a position with an odd number of edges. The player
who moves to a position with an even number of edges must always win.
This means that for $G=(V,E)$ we have $G \in \mathcal{P}$ if $|E|$ is
even and $G \in \mathcal{N}$ is $|E|$ is odd.

Since \emph{every} move from any $\mathcal{N}$-position takes you to a
$\mathcal{P}$-position we cannot get any Grundy value larger than one.
The assertion above follows.
\end{proof}

\begin{cor}
\cite[Conjecture 17]{shelton}
The Grundy value of an odd/odd vertex removal game played on a grid
graph is either 0 or 1.
\end{cor}

\subsection{Graphs for every Grundy value}

Another interesting question concerning odd/odd vertex removal games
is whether there are graphs for every possible Grundy value. This question
was first posed by Ottaway\cite{ottaway}. Shelton\cite{shelton} went even
further and made a conjecture which would answer this question even
for connected graphs.

We have already seen graphs of Grundy value 1 and 0. Ottaway gives the
following two examples of graphs with Grundy value 2:
\begin{figure}[H]
 \centering
 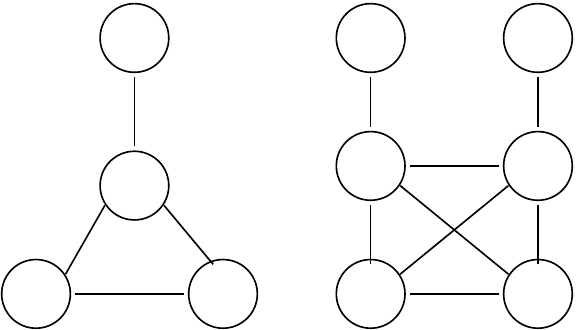 
 \caption{Odd/odd vertex removal positions with Grundy value 2.}
 \label{figure:gv2s}
\end{figure} 

We will prove that given a set of graphs $G_0,G_1,...,G_n$ such that
$g(G_i) = i$ we can constrct a new graph $G$ such that $g(G) = n+1$ and
thus the existance of graphs for every Grundy value follows by induction. We
will use another construction than the one considered by Shelton in
\cite[Conjecture 16]{shelton}.

The main idea is that the game where the first player decides on which of the
graphs $G_0,G_1,...,G_n$ to play the game as the first move is a game $G$ with
Grundy value given by
$$g(G) = \mex\{g(G_0),g(G_1),...,g(G_n)\} = n+1$$
we will find a graph which simulates this choice of graph to play by forcing
the removal of one of $n+1$ vertices each leaving all but one of the $G_i$s
with only even degree vertices.

\begin{lemma}
\label{lollemma}
Let $G$ be a graph. Let $G'$ be the graph you obtain if you for every
vertex of degree 0 replace it by $P_3$ (the path on 3 vertices). Then
$$g(G) = g(G')$$
\end{lemma}
\begin{proof}
Since a vertex of degree $0$. Since each component of the graph on which
one playes the game is disjunctive the graph $G$ can be thought of as the
disjunctive sum of its components, $G = C_1 + C_2 + ... + C_n$. Then by Theorem
\ref{mainthm} we know that
$$g(G) = g(C_1) \oplus g(C_2) \oplus ... \oplus g(C_n)$$
thus the assertion holds if $g(C_i) = g(P_3)$ for the one-vertex components,
but this is true since the both the one-vertex graph and the path on 3
vertices has Grundy value 0.
\end{proof}

The above lemma means that when we claim that (in the induction assumption)
that there are graphs $G_0,G_1,...,G_n$ for every Grundy value in $\{0,1,...,n\}$
we can take such graphs that all have at least two odd vertices. Every
graph with $g(G) > 0$ cannot have only even vertices by not being terminal
positions in the game and every graph has an even number of odd vertices
\ref{diestel}.

We may even assume that each of the graphs $G_0,G_1,...,G_n$ are connected
since we will construct a graph with Grundy-value $n+1$ which is connected.

\begin{thm}
For every $k \in \mathbb{N}$ there is a connected graph, G, such that
$$g(G) = k$$
\end{thm}
\begin{proof}
We will prove this by induction on $k$. We have already seen examples of
connected graphs for $k \leq 2$.

Suppose that there are connected graphs with Grundy value $k$ for every
$k \leq K$. Take a collection of connected graphs $G_0,G_1,...,G_K$ such that
$g(G_i) = i$ for all $i = 0,1,...,K$. Moreover we may by Lemma \ref{lollemma}
assume that that the number of odd vertices in each of these graphs is even
and not less than 2.

We construct a new graph by taking the disjoint sum (can be seen as the
disjoint union of the sets of vertices and edges) of all the graphs
$G_0,G_1,...,G_K$ and then adding new vertices $v_0,v_1,...,v_K$. We connect
each $v_i$ to every odd degree vertex in $G_i$. In the case when $K$ is
even we add another $v_{-1}$ together with $G_{-1} := P_3$ and connect this
in the same way.

The set of vertices $\{v_i\}$ now consists of an even number of vertices
all having even degree. We can can add all possible edges between the
vertices in that set (so that they induce a complete graph). Then since
vertices of $K^{2n}$ always have odd degree so does the $\{v_i\}$. Call the
graph constructed in this way $G = (V,E)$.

In explicit terms; given $G_{-1} = P_3, G_0, ..., G_K$ such that $g(G_i) = i$ for $i=0,1,...,K$ and we write $G_i = (V_i,E_i)$ for $i=-1,0,...,K$, let
\begin{align*}
V &= \bigsqcup_{i \in I} V_i \cup \{v_i\}_{i\in I} \\
E &= \{v_iv_j | i < j, i,j \in I\} \cup \{v_i v | \deg(v) \textnormal{ even in }
  G_i,\; v \in V_i\} \cup \bigsqcup_{i \in I} E_i \\
\textnormal{where } I &= \{-1,0,1,...,K\} \quad \textnormal{ if } K \textnormal{ even} \\
      I &= \{0,1,...,K\} \quad \textnormal{ if } K \textnormal{ odd} 
\end{align*}
then we have $G = (V,E)$. It remains to prove that indeed $g(G) = K+1$.

Since each odd vertex of each graph $G_i$ have one additional neighbor in
$G$, all of the original vertices have even degree in $G$ (the original
even degree vertices are unchanged). On the other hand each \textsl{new}
vertex $v_i$ has even degree.

The Grundy value $g(G-v_i) = g(G_i)$ since selecting the $v_i$ gives us two
disjoint components. One consisting of the $G_j$ where $j \neq i$ together with the $v_j$ corresponding to these. For each vertex $v_j$ in $G-v_i$ we will
have the $\deg v_j$ is even, by all of these being neighbors to $v_i$. This component
has no playable vertices and thus has Grundy-value $0$. The other component
is by cunstruction just $G_i$. Thus $g(G-v_i) = 0 \oplus g(G_i) = g(G_i)$.

Now since we only have odd degree for the vertices $v_i, \; i \in I$ we
also get
$$g(G) = \mex \{g(G_i) : i \in I\} = K + 1$$
by selection of the $G_i$:s.

By induction the theorem follows.
\end{proof}

\section{Conclusions and future work} 
\label{conclusion}


We have successfully proved a Conjecture of Shelton \cite{shelton} and
also we have affirmatively answered a question about the existance of
games for each Grundy-value which was posed by both Shelton\cite{shelton}
and Ottaway\cite{ottaway}. In actuality we prove a more general result about
the Grundy value of bipartite graphs rather than only grid graphs.

We do this by another construction than the one proposed by Shelton in his
conjecture. The essensial question of the existance of games for each
Grundy value has been answered, so there would be less use of proving
Shelton's second Conjecture, unless it could lead us to some insight into
the problem of determining Grundy values for other classes of graphs.

\subsection{Open problems}

The main problem of constructing a function for computing the Grundy value
of any graph, i.e. finding $g(G)$ for general graphs and not just simple
classes of graphs such as bipartite which is done in this paper, is still
an open problem. However the neither the result or the method from this
paper will be of much use for proving that. This is because we use specific
properties of the cycles in bipartite graphs and prove that these games
are much simpler than the one played on arbitrary graphs (it does not
matter what move you make). Hence, when solving this problem one will have
to use some other method than the one for bipartite graphs.

Even though the problem of determine the Grundy-value in a directed graph has
been shown to PSPACE-complete\cite{shelton}, there is no reason to dispair.
The problem in a directed graph might be much more complicated than in
the undirected case (as illustrated by the Even/even vertex removal game).
There is still hope for some simple graph invariant that characterizes
the games in terms of which player wins or the Grundy-value.

While the problem of Shelton's second conjecture is still open we have proved
that in fact there are graphs for every Grundy value and thus anwered the
underlying question by another construction. Proving this might however give
some more insight into how to determine the Grundy-value of other classes
of graphs.

\bibliography{bibl}{}
\bibliographystyle{plain} 

\end{document}

%% file: examples.pdf_t
\begin{picture}(0,0)%
\includegraphics{examples.pdf}%
\end{picture}%
\setlength{\unitlength}{3947sp}%
\begingroup\makeatletter\ifx\SetFigFont\undefined%
\gdef\SetFigFont#1#2#3#4#5{%
  \reset@font\fontsize{#1}{#2pt}%
  \fontfamily{#3}\fontseries{#4}\fontshape{#5}%
  \selectfont}%
\fi\endgroup%
\begin{picture}(5671,1901)(203,-1114)
\put(4876,-1036){\makebox(0,0)[lb]{\smash{{\SetFigFont{12}{14.4}{\familydefault}{\mddefault}{\updefault}{\color[rgb]{0,0,0}$g(K_{3,2}) = 0$}%
}}}}
\put(3676,-1036){\makebox(0,0)[lb]{\smash{{\SetFigFont{12}{14.4}{\familydefault}{\mddefault}{\updefault}{\color[rgb]{0,0,0}$g(S_6) = 1$}%
}}}}
\put(2026,-1036){\makebox(0,0)[lb]{\smash{{\SetFigFont{12}{14.4}{\familydefault}{\mddefault}{\updefault}{\color[rgb]{0,0,0}$g(K_4) = 1$}%
}}}}
\put(376,-1036){\makebox(0,0)[lb]{\smash{{\SetFigFont{12}{14.4}{\familydefault}{\mddefault}{\updefault}{\color[rgb]{0,0,0}$g(P_5) = 0$}%
}}}}
\end{picture}%

%% file: gv2s.pdf_t
\begin{picture}(0,0)%
\includegraphics{gv2s.pdf}%
\end{picture}%
\setlength{\unitlength}{4144sp}%
\begingroup\makeatletter\ifx\SetFigFont\undefined%
\gdef\SetFigFont#1#2#3#4#5{%
  \reset@font\fontsize{#1}{#2pt}%
  \fontfamily{#3}\fontseries{#4}\fontshape{#5}%
  \selectfont}%
\fi\endgroup%
\begin{picture}(2625,1500)(61,-676)
\end{picture}%